\date{October 7, 2009}

\documentclass[11pt]{amsart}
\usepackage{latexsym,amsmath,amsfonts,amscd,amssymb}
\usepackage{graphics}
\newtheorem{theorem}{Theorem}[section]
\newtheorem{lemma}[theorem]{Lemma}
\newtheorem{proposition}[theorem]{Proposition}

\newtheorem{definition}[theorem]{Definition}
\newtheorem{conjecture}[theorem]{Conjecture}

\theoremstyle{remark}
\newtheorem{remark}[theorem]{Remark}

\newcommand{\la}{\langle}
\newcommand{\ra}{\rangle}

\newcommand{\inc}{\hookrightarrow}
\newcommand{\bd}{\partial}
\newcommand{\x}{\times}

\newcommand{\Vol}{\operatorname{Vol}}

\newcommand{\Int}{\operatorname{Int}}

\newcommand{\id}{\operatorname{id}}

\newcommand{\cA}{{\mathcal A}}

\newcommand{\cC}{{\mathcal C}}

\newcommand{\cM}{{\mathcal M}}

\newcommand{\cL}{{\mathcal L}}

\newcommand{\QQ}{{\mathbb Q}}
\newcommand{\RR}{{\mathbb R}}
\newcommand{\TT}{{\mathbb T}}

\newcommand{\ZZ}{{\mathbb Z}}

\renewcommand{\a}{\alpha}

\newcommand{\g}{\gamma}

\renewcommand{\t}{\tau}

\newcommand{\G}{\Gamma}

\begin{document}

\title{Ergodic solenoidal homology: Realization theorem}


\author[V. Mu\~{n}oz]{Vicente Mu\~{n}oz}
\address{Instituto de Ciencias Matem\'aticas
CSIC-UAM-UC3M-UCM, Serrano 113 bis, 28006 Madrid, Spain}

\address{Facultad de
Matem\'aticas, Universidad Complutense de Madrid, Plaza de Ciencias
3, 28040 Madrid, Spain}

\email{vicente.munoz@imaff.cfmac.csic.es}

\author[R. P\'{e}rez Marco]{Ricardo P\'{e}rez Marco}
\address{CNRS, LAGA UMR 7539, Universit\'e Paris XIII, 
99, Avenue J.-B. Cl\'ement, 93430-Villetaneuse, France}


\email{ricardo@math.univ-paris13.fr}

\thanks{Partially supported through Spanish MEC grant MTM2007-63582.
Second author supported by CNRS (UMR 7539).}

\maketitle

\begin{abstract}
We define generalized currents associated with immersions of
abstract oriented solenoids with a transversal measure. We realize
geometrically the full real homology of a compact manifold with
these generalized currents, and more precisely with immersions of
minimal uniquely ergodic solenoids. This makes precise and geometric
De Rham's realization of the real homology by only using a
restricted geometric subclass of currents.
\end{abstract}

\section{Introduction} \label{sec:introduction}

We consider a smooth compact connected oriented manifold $M$ of
dimension $n\geq 1$. Any closed oriented submanifold $N\subset M$
of dimension $0\leq k\leq n$ determines a homology class in
$H_k(M, \ZZ)$. This homology class in $H_k(M,\RR)$, as dual of De
Rham cohomology, is explicitly given by integration of the
restriction to $N$ of differential $k$-forms on $M$. Also, any
immersion $f:N \to M$ defines an integer homology class in a
similar way by integration of pull-backs of $k$-forms.
Unfortunately, because of topological reasons dating back to Thom
\cite{Thom1} \cite{Thom2}, not all integer homology classes in
$H_k(M,\ZZ )$ can be realized in such a way. Geometrically, we can
realize any class in $H_k(M, \ZZ)$ by topological $k$-chains. The
real homology $H_k(M,\RR)$ classes are only realized by formal
combinations with real coefficients of $k$-cells. This is not
satisfactory for various reasons. In particular, for diverse
purposes it is important to have an explicit realization, as
geometric as possible, of real homology classes.

\medskip

The first contribution in this direction came in 1957 from the work
of S. Schwartzman \cite{Sc}. Schwartzman showed how, by a limiting
procedure, one-dimensional curves embedded in $M$ can define a real
homology class in $H_1(M,\RR)$. More precisely, he proved that this
happens for almost all curves solutions to a differential equation
admitting an invariant ergodic probability measure. Schwartzman's
idea is very natural. It consists on integrating $1$-forms over
large pieces of the parametrized curve and normalizing this integral
by the length of the parametrization. Under suitable conditions, the
limit exists and defines an element  of the dual of $H^1(M,\RR )$,
i.e. an element of $H_1(M, \RR )$. This procedure is equivalent to
the more geometric one of closing large pieces of the curve by
relatively short closing paths. The closed curve obtained defines an
integer homology class. The normalization by the length of the
parameter range provides a class in $H_k(M,\RR )$. Under suitable
hypothesis, there exists a unique limit in real homology when the
pieces exhaust the parametrized curve, and this limit is independent
of the closing procedure. In the article \cite{MPM3}, we study
the different aspects of the Schwartzman procedure, that we extend
to higher dimension.

\medskip

Later in 1975, D. Ruelle and D. Sullivan \cite{RS} defined, for
arbitrary dimension $0\leq k\leq n$, geometric currents by using
oriented $k$-laminations embedded in $M$ and endowed with a
transversal measure. They applied their results to stable and unstable
laminations of Axiom A diffeomorphisms. In a later article Sullivan \cite{Su} extended
further these results and their applications. The point of view of
Ruelle and Sullivan is also based on duality. The observation is
that $k$-forms can be integrated on each leaf of the lamination and
then all over the lamination using the transversal measure. This
makes sense locally in each flow-box, and then it can be extended
globally by using a partition of unity.  The result only depends on
the cohomology class of the $k$-form. In \cite{MPM2}
we review and extend Ruelle-Sullivan theory.

\medskip

%
%
%
%
%
%

It is natural to ask whether it is possible to realize every real
homology class using a topologically minimal (i.e. all leaves dense) 
Ruelle-Sullivan current. In order to achieve
this goal we must enlarge the class of Ruelle-Sullivan currents by
considering immersions of abstract oriented solenoids. We
define a $k$-solenoid to be a Hausdorff compact space foliated by
$k$-dimensional leaves with finite dimensional transversal structure
(see the precise definition in section \ref{sec:minimal}). For these
oriented solenoids we can consider $k$-forms that we can integrate
provided that we are given a transversal measure invariant by the
holonomy group. We define an immersion of a solenoid $S$ into $M$ to
be a regular map $f: S\to M$ that is an immersion in each leaf. If
the solenoid $S$ is endowed with a transversal measure $\mu$, then
any smooth $k$-form in $M$ can be pulled back to $S$ by $f$ and
integrated. The resulting numerical value only depends on the
cohomology class of the $k$-form. Therefore we have defined a closed
current that we denote by $(f,S_\mu )$ and that we call a {\it generalized
current}. This gives a homology class
$[f,S_\mu] \in H_k(M,\RR )$. Our main result is:

\begin{theorem}\textbf{\em (Realization Theorem)} \label{thm:1.2}
Every real homology class in $H_k(M,\RR )$ can be realized by a
generalized current $(f,S_\mu)$ where $S_\mu$ is an oriented,
minimal, uniquely ergodic solenoid.
\end{theorem}

Minimal and uniquely ergodic solenoids are defined later on. This
result strengths De Rham's realization theorem of homology classes by abstract currents, i.e. forms with coefficients distributions. It is a geometric De Rham's Theorem where the abstract currents are replaced by generalized currents that are geometric objects.

\medskip

We can ask why we do need to enlarge the class of Ruelle-Sullivan currents.
The result does not hold for minimal Ruelle-Sullivan currents due to the following observation from \cite{MPM2}
(compare with \cite{HM}).

\begin{proposition} \label{thm:11}
Homology classes with non-zero self-intersection cannot be
represented by Ruelle-Sullivan currents with no compact leaves.
\end{proposition}

Therefore it is not possible to represent a real homology
class in $H_k(M,\RR )$ with non-zero self-intersection by a minimal 
Ruelle-Sullivan current that is not a submanifold. Note that this
obstruction only exists when $n-k$ is even. This may be the
historical reason behind the lack of results on the representation
of an arbitrary homology class by minimal Ruelle-Sullivan currents.

\medskip

The space of solenoids is large, and we would like to realize
the real homology classes by a minimal class of solenoids enjoying
good properties. We are first naturally led to topological
minimality. As we prove in \cite{MPM2}, the spaces of
$k$-solenoids is inductive and therefore there are always minimal
$k$-solenoids. However, the transversal structure and the holonomy
group of minimal solenoids can have a rich structure. In particular,
such a solenoid may have many distinct transversal measures, each
one yielding a different generalized current for the same immersion
$f$. Also when we push Schwartzman ideas beyond $1$-homology for
some nice classes of solenoids, we see that in general, even when
the immersion is an embedding, the generalized current does not
necessarily coincide with the Schwartzman homology class of the
immersion of each leaf (actually not even this Schwartzman class
needs to be well defined). Indeed the classical literature lacks of
information about the precise relation between Ruelle-Sullivan and
Schwartzman currents. One would naturally expect that there is some
relation between the generalized currents and the Schwartzman
current (if defined) of the leaves of the lamination. We study this
problem in \cite{MPM3}.
The main result is that there is such relation for the class of minimal,
ergodic solenoids with a trapping region.
A solenoid with a trapping region (see definition in section \ref{sec:minimal})
has holonomy group generated by a single
map. Then the bridge between generalized currents and Schwartzman
currents of the leaves is provided by Birkhoff's ergodic theorem. The
main result of \cite{MPM3} is the following.

\begin{theorem}\label{thm:1.3}
Let $S_\mu$ be a minimal solenoid endowed with an
ergodic transversal measure $\mu$ and possessing a trapping region $W$.
Let $f: S_\mu \to M$ be an
immersion of $S_\mu$ into $M$ such that $f(W)$ is contained in a ball of $M$.
Then for $\mu$-almost all leaves
$l\subset S_\mu$, the Schwartzman homology class of $f(l)\subset M$
is well defined and coincides with the homology class
$[f,S_\mu]$.

If moreover $S$ is uniquely ergodic, then this happens for {\it all} leaves.
\end{theorem}

(We recall the definition of Schwartzman homology class and trapping region in
section \ref{sec:minimal}.)

The solenoids constructed for the proof of the Realization Theorem do satisfy the hypothesis 
of this theorem and the transversal measure is unique, that is, the solenoids are uniquely 
ergodic.

\bigskip

{\bf Solenoidal Hodge Conjecture.}

\bigskip

The Hodge Conjecture is an statement about the geometric realization of an integral
class of pure type $(p,p)$ in a complex (projective) manifold. If we drop the condition
of the class being integral, then theorem \ref{thm:1.2} suggests a natural
conjecture for {\it real} homology classes of pure type as follows.

For a compact K\"ahler manifold $M$ of complex dimension $n$, a
complex immersed solenoid $f:S_\mu\to M$ (that is, a solenoid
where the images $f(l)$ of the leaves $l\subset S_\mu$ are complex
immersed submanifolds), of dimension $k=2(n-p)$, defines a class
in $H_{n-p,n-p}(M)=H^{p,p}(M)^*\subset H_k(M,\RR)$, as proved
in proposition 7.8 of \cite{MPM2}. It is natural
to formulate the following conjecture:

\begin{conjecture}\textbf{\em (Solenoidal Hodge Conjecture)}
\label{conj:SHC} Let $M$ be a compact K\"ahler manifold. Then any
class in $H^{p,p}(M)$ is represented by a complex immersed
solenoid of dimension $k=2(n-p)$.
\end{conjecture}

Note that the standard Hodge Conjecture is stated for projective
complex manifolds, since it fails for K\"ahler manifolds
\cite{Zu}. The counterexamples of \cite{Zu} are non-algebraic
complex tori. It is easy to see that conjecture \ref{conj:SHC}
holds for complex tori (using non-minimal complex solenoids).

\noindent \textbf{Acknowledgements.} \
The authors are grateful to Alberto Candel, Etienne Ghys, Nessim Sibony,
Dennis Sullivan
and Jaume Amor\'os
for their comments and interest on this work. In particular, Etienne
Ghys early pointed out on the impossibility of realization in
general of integer homology classes by embedded manifolds.

The first author wishes to acknowledge Universidad Complutense de
Madrid and Institute for Advanced Study at Princeton for their
hospitality and for providing excellent working conditions.  The second author
thanks Jean Bourgain and the IAS at Princeton for their hospitality
and facilitating the collaboration of both authors.

\section{Solenoids and generalized currents} \label{sec:minimal}

Let us review the main concepts introduced in \cite{MPM2}.

\begin{definition}\label{def:k-solenoid} A
$k$-solenoid, where $k\geq 0$, of class $C^{r,s}$, is a compact Hausdorff space endowed with an atlas
of flow-boxes $\cA=\{ (U_i,\varphi_i)\}$,
 $$
 \varphi_i:U_i\to D^k\x K(U_i)\, ,
 $$
where $D^k$ is the $k$-dimensional open ball, and $K(U_i)\subset \RR^l$ is the transversal
set of the flow-box. The changes of charts $\varphi_{ij}=\varphi_i\circ
\varphi_j^{-1}$ are of the form
 \begin{equation}\label{eqn:change-of-charts}
 \varphi_{ij}(x,y)=(X(x,y), Y(y))\, ,
 \end{equation}
where $X(x,y)$ is of class $C^{r,s}$ and $Y(y)$ is of class $C^s$.
\end{definition}

Let $S$ be a $k$-solenoid, and $U\cong D^k \x K(U)$ be a flow-box for $S$. The sets
$L_y= D^k\x \{y\}$ are called the (local) leaves of the flow-box. A leaf $l\subset S$ of the
solenoid is a connected $k$-dimensional manifold whose intersection with any flow-box
is a collection of local leaves. The solenoid is oriented if the leaves are oriented
(in a transversally continuous way).

A transversal for $S$ is a subset $T$ which is a finite union of transversals of flow-boxes.
Given two local transversals $T_1$ and $T_2$ and
a path contained in a leaf from a point of $T_1$ to a point of $T_2$,
there is a well-defined holonomy map $h:T_1\to T_2$. The holonomy maps form a pseudo-group.

A $k$-solenoid $S$ is minimal if it does not contain a proper sub-solenoid. By \cite[section 2]{MPM2},
minimal sub-solenoids do exist in any solenoid. If $S$ is minimal, then any transversal is a global transversal, i.e., it
intersects all leaves. In the special case of an oriented minimal $1$-solenoid, the holonomy
return map associated to a local transversal,
 $$
 R_T:T \to T
 $$
is known as the Poincar\'e return map (see \cite[Section 4]{MPM2}).

\begin{definition} \label{def:transversal-measure}
Let $S$ be a $k$-solenoid. A transversal measure $\mu=(\mu_T)$ for
$S$ associates to any local transversal $T$ a locally finite measure
$\mu_T$ supported on $T$, which are invariant by the holonomy
pseudogroup, i.e. if $h : T_1 \to T_2$ is a holonomy map, then
$h_* \mu_{T_1}= \mu_{T_2}$.
\end{definition}

We denote by $S_\mu$ a $k$-solenoid $S$ endowed with a transversal
measure $\mu=(\mu_T)$. We refer to $S_\mu$ as a measured solenoid.
Observe that for any transversal measure $\mu=(\mu_T)$ the scalar
multiple $c\, \mu=(c \, \mu_T)$, where $c>0$, is also a transversal
measure. Notice that there is no natural scalar normalization of
transversal measures.

\begin{definition} \label{def:transversal-ergodicity}\textbf{\em (Transverse ergodicity)}
A transversal measure $\mu=(\mu_T )$ on a solenoid $S$ is ergodic if for any Borel set
$A\subset T$ invariant by the pseudo-group of holonomy maps on $T$,
we have
 $$
 \mu_T(A) = 0 \ \ {\hbox{\rm{ or }}} \ \ \mu_T(A) = \mu_T(T) \, .
 $$
We say that $S_\mu$ is an ergodic solenoid.
\end{definition}

\begin{definition} \label{def:transversal-unique-ergodicity}
Let $S$ be a $k$-solenoid. The solenoid $S$ is
uniquely ergodic if it has a unique (up to scalars)
transversal measure $\mu$ and its support is the whole of $S$.
\end{definition}

\bigskip

Now let $M$ be a smooth manifold of dimension $n$. An immersion of a $k$-solenoid
$S$ into $M$, with $k<n$, is a smooth map $f:S\to M$ such that the differential
restricted to the tangent spaces of leaves has rank $k$ at every
point of $S$. The solenoid $f:S\to M$ is
transversally immersed if for any flow-box $U\subset S$ and chart
$V\subset M$, the map $f:U= D^k\x K(U) \to V\subset \RR^n$ is
an embedding, and the images of
the leaves intersect transversally in $M$. If moreover $f$ is injective, then
we say that the solenoid is embedded.

Note that under a transversal immersion, resp.\ an embedding,
$f:S\to M$, the images of the leaves are immersed, resp.\
injectively immersed, submanifolds.

\bigskip

\begin{definition}\label{def:Ruelle-Sullivan}\textbf{\em (Generalized currents)}
Let $S$ be an oriented  $k$-solenoid of class $C^{r,s}$, $r\geq
1$, endowed with a transversal measure $\mu=(\mu_T)$. An immersion
 $$
 f:S\to M
 $$
defines a current $(f,S_\mu)\in \cC_k(M)$, called generalized Ruelle-Sullivan
current (or just generalized current), as follows.
Let $\omega$ be an $k$-differential form in $M$. The pull-back
$f^* \omega$ defines a $k$-differential form on the leaves of $S$.
Let $S=\bigcup_i S_i$ be a measurable partition such that each
$S_i$ is contained in a flow-box $U_i$.  We define
 $$
 \la (f,S_\mu),\omega \ra=\sum_i \int_{K(U_i)} \left ( \int_{L_y\cap S_i}
 f^* \omega \right ) \ d\mu_{K(U_i)} (y) \, ,
 $$
where $L_y$ denotes the horizontal disk of the flow-box.

The current $(f,S_\mu)$ is closed, hence it defines a
real homology class
 $$
 [f,S_\mu]\in H_k(M,\RR )\, ,
 $$
called Ruelle-Sullivan homology class.
\end{definition}

Note that this definition does not depend on the measurable
partition (given two partitions consider the common refinement).
If the support of $f^*\omega$ is contained in a flow-box $U$ then
 $$
 \la (f,S_\mu),\omega \ra =\int_{K(U)} \left ( \int_{L_y} f^* \omega \right )
 \ d\mu_{K(U)} (y) \, .
 $$
In general, take a partition of unity $\{\rho_i\}$ subordinated to
the covering $\{U_i\}$, then
  $$
   \la (f,S_\mu),\omega \ra = \sum_i
   \int_{K(U_i)} \left( \int_{L_y} \rho_i f^* \omega \right)
   d\mu_{K(U_i)} (y) \, .
  $$

Let us see that $(f,S_\mu)$ is closed.
For any exact differential $\omega=d\a$ we have
 $$
 \begin{aligned}
 \la (f,S_\mu),d\a\ra =\, &  \sum_i
 \int_{K(U_i)} \left ( \int_{L_y} \rho_i \, f^* d\a \right )
 \ d\mu_{K(U_i)}(y) \\  =\, &  \sum_i
 \int_{K(U_i)} \left ( \int_{L_y} d (\rho_i f^* \a) \right )
 \ d\mu_{K(U_i)}(y)   \\ & \qquad  -
  \sum_i \int_{K(U_i)} \left ( \int_{L_y} d \rho_i \wedge f^* \a \right )
 \ d\mu_{K(U_i)}(y)    = 0\, . \qquad
 \end{aligned}
 $$
The first term vanishes using Stokes in each leaf (the form $\rho_i f^* \a$
is compactly
supported on $U_i$), and the second term vanishes because $\sum_i d\rho_i\equiv 0$.
Therefore $[f, S_\mu]$ is a well defined homology class of degree $k$.

In their original article \cite{RS}, Ruelle and Sullivan defined
this notion for the restricted class of solenoids embedded in $M$.

%
%
%

When $M$ is a compact and oriented $n$-manifold, the Ruelle-Sullivan
homology class $[f,S_\mu]\in H_k(M,\RR)$ gives an element
 $$
 [f,S_\mu]^* \in H^{n-k}(M,\RR)\, ,
 $$
under the Poincar\'{e} duality isomorphism $H_k(M,\RR)\cong
H^{n-k}(M,\RR)$.

We have the following result (theorem 7.6 in \cite{MPM2}) which proves Proposition \ref{thm:11}.

\begin{theorem}\label{thm:self-intersection} \textbf{ \em
(Self-intersection of embedded solenoids)} Let $M$ be a compact,
oriented, smooth manifold. Let $f:S_\mu\to M$ be an embedded oriented measured
solenoid, such that the transversal measures
$(\mu_T)$ have no atoms. Then we have
 $$
 [f,S_\mu]^* \cup [f,S_\mu]^*=0
 $$
in $H^{2(n-k)}(M,\RR)$.
\end{theorem}

This indicates that we cannot use only {\em embedded} solenoids to represent
real homology classes in general.

\bigskip

Now let us recall the notions of Schwartzman theory that we are going
to need, and that are extensively studied in \cite{MPM3}.

Let $M$ be a compact smooth Riemannian manifold.
Given an Riemannian immersion
$c: N \to M$ from an oriented complete smooth manifold $N$ of dimension
$k\geq 1$, we consider exhaustions $(U_n)$ of $N$ with
$U_n\subset N$ being $k$-dimensional compact submanifolds with
boundary $\bd U_n$. We close $U_n$ with a $k$-dimensional oriented
manifold $\G_n$ with boundary $\bd \G_n=-\bd U_n$ (that is, $\bd
U_n$ with opposite orientation, so that $N_n=U_n\cup \G_n$ is a
$k$-dimensional compact oriented manifold without boundary), in such
a way that $c_{|U_n}$ extends to a piecewise 
smooth map $c_n:N_n \to M$. We may consider the associated homology
class $[c_n(N_n)]\in H_k(M,\ZZ)$. Suppose that
 $$
 \frac{\Vol_k (c_n(\Gamma_n) )}{\Vol_k (c_n(N_n) )}\to 0\, .
  $$
If the following limit exists,
 \begin{equation}\label{eqn:extra}
 \lim_{n\to +\infty } \frac{1}{\Vol_k (c_n(N_n) )} [c_n(N_n)] \in
 H_k(M,\RR) \, ,
 \end{equation}
we call it a Schwartzman asymptotic $k$-cycle.

\begin{definition}\label{def:regular-asymptotic-k-cycles}
The immersed manifold $c:N\to M$ represents a homology class $a\in
H_k(M,\RR )$ if for all exhaustions $(U_n)$, the class (\ref{eqn:extra})
exists and equals $a$.
We denote $[{c,N}]=a$, and call it the Schwartzman homology class of $(c,N)$.
\end{definition}

For immersed solenoids $f:S\to M$, we may consider the Schwartzman homology
classes associated to its leaves.

\begin{definition} \textbf{\em (Schwartzman representation of homology classes)}
\label{def:representation-Schwartzman-k-solenoid} Let $f:S_\mu\to M$ be an
immersion in $M$ of an oriented measured $k$-solenoid $S$, and give $S$ the
induced Riemannian structure. 
The immersed solenoid $f:S_\mu\to M$
fully represents a homology
 class $a\in H_1(M, \RR)$ if for all leaves $l\subset S$,
 we have that $(f,l)$ is a
 Schwartzman asymptotic $k$-cycle with $[{f,l}]=a$.
\end{definition}

A class of solenoids with good properties are those which have a trapping region,
since for them the holonomy is represented by a single map. The definition is
cumbersome but very natural.

\begin{definition}\textbf{\em (Trapping region)} \label{def:trapping}
An open subset $W\subset S$ of a solenoid $S$ is a trapping region
if there exists a continuous map $\pi : S \to \TT$ such that
\begin{enumerate}
\item[(1)] For some $0<\epsilon_0<1/2$, $W=\pi^{-1} ((-\epsilon_0 ,\epsilon_0))$.

\item[(2)]  There is a global transversal $T\subset \pi^{-1} (\{ 0\} )$.

\item[(3)] Each connected component of $\pi^{-1} (\{ 0\} )$ intersects $T$ in exactly
one point.

\item[(4)] $0$ is a regular value for $\pi$. 

\item[(5)] For each connected component $L$ of $\pi^{-1} (\TT-\{0\} )$ we
have ${\overline L}\cap T=\{ x,y\}$, where $\{x\} \in {\overline L}\cap T\cap
\pi^{-1} ((-\epsilon_0 ,0])$ and $\{y\} \in {\overline L}\cap T\cap \pi^{-1}
([0,\epsilon_0 ))$. 
\end{enumerate}

\end{definition}

The main result of \cite{MPM3} is the following theorem.

\begin{theorem} \label{thm:11.12}
Let $S$ be a minimal oriented $k$-solenoid endowed with a transversal uniquely ergodic
measure $\mu\in\cM_\cL(S)$ and with a trapping region $W\subset S$.
Consider an immersion $f:S\to M$ such that $f(W)$ is contained in a
contractible ball in $M$. Then $f:S_\mu \to M$ fully represents its
Ruelle-Sullivan homology class $[f,S_\mu ]$.
\end{theorem}

\section{Realization of $H_1(M,\RR)$ } \label{sec:1-solenoids}

Let $M$ be a $C^\infty$ smooth compact Riemannian manifold. Given a
real $1$-homology class $a\in H_1(M,\RR)$, we want to construct an
immersion $f:S\to M$ in $M$ of a uniquely ergodic solenoid $S_\mu$
with
generalized current $[f,S_\mu]=a$.

In some situations (depending on the dimension) we will achieve an
embedding. Actually the abstract $1$-solenoid $S$ that we will construct is
independent of $a$ and of $M$, and moreover it has a $1$-dimensional
transversal structure.

Let $h:\TT\to \TT$ be a diffeomorphism of the circle with an irrational rotation
number (and therefore uniquely ergodic), which is a Denjoy
counter-example, i.e. has the unique invariant probability measure
supported in the Cantor set $K\subset \TT$. Let $\mu_K$ denote the
invariant probability measure. For the original construction of
Denjoy counter-examples see  \cite{Denjoy}. Actually for any given $\epsilon >0$,
$h$ can be
taken to be of class $C^{2-\epsilon}$ (see
\cite{Herman}).

The suspension of $h$,
  $$
  S_h= ([0,1]\x \TT )_{/ (0,x) \sim (1,h(x))}
  $$
is $C^{2-\epsilon}$-diffeomorphic to the $2$-torus $T^2$. More
explicitly, the diffeomorphism is as follows: take $c>0$ small, let
$h_t$, $t\in [0,c]$, be a (smooth) isotopy from $\id$ to $h$,
then we define the diffeomorphism $H:T^2\to S_h$ by
 $$
   H(t,x)=\left\{ \begin{array}{ll} (t,h^{-1}(h_t(x))),
   \qquad &\text{for $t\in [0,c]$\,,} \\
   (t,x), & \text{for $t\in [c,1]$\,.} \end{array} \right.
 $$
Note that $S_h$ is foliated by the horizontal leaves, so $T^2$ is
foliated accordingly. It can be considered also as a $1$-solenoid
of class $C^{\omega,2-\epsilon}$.

The sub-solenoid
  $$
  S= ( [0,1]\x K)_ {/ \sim} \, \subset S_h
  $$
is an oriented $1$-solenoid of class $C^{\omega,2-\epsilon}$, with
transversal $T= (\{0\}\x \TT ) \cap S= \{0\} \x  K$. The holonomy is
given by the map $h$, which is uniquely ergodic. Moreover, the
associated transversal measure is $\mu_K$ on the transversal $K\cong
\{0\}\x K$. So $S$ is an oriented and uniquely ergodic $1$-solenoid.

Using the diffeomorphism $H$, we may see the solenoid $S$ inside
the $2$-torus, $S\subset S_h\cong T^2$, consisting of the paths
$(t,x)$, $x\in K$, $t\in [c,1]$, together with the paths
$(t,h_t(x))$, $x\in K$, $t\in [0,c]$. The embedding $S \inc T^2$
is of class $C^{\omega,2-\epsilon}$, so we shall think of $S$ as
an oriented $1$-solenoid of regularity $C^{\omega,2-\epsilon}$.

\begin{figure}[h]\label{figure1}
\centering
\resizebox{6cm}{!}{\includegraphics{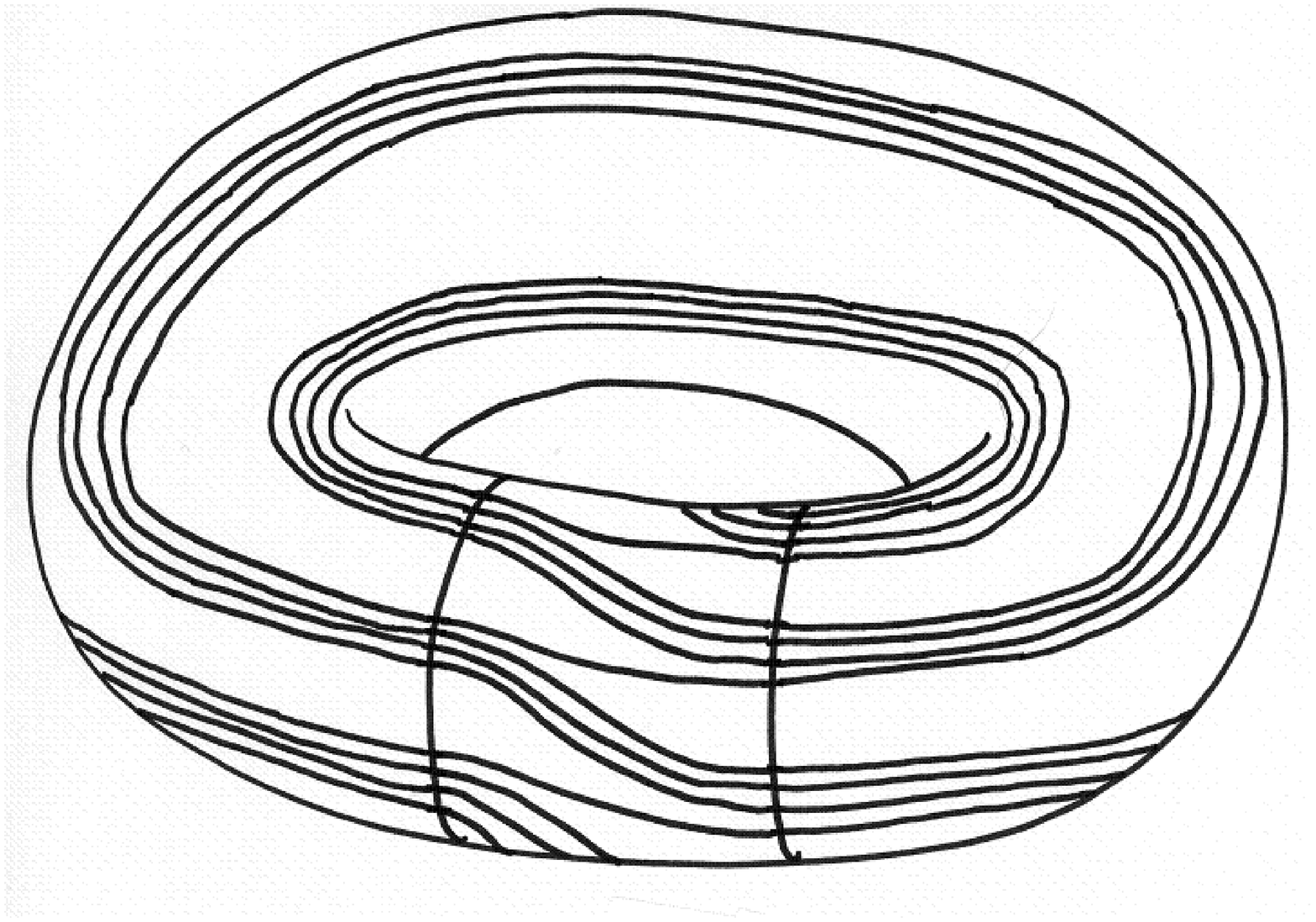}}    
\caption{The $1$-solenoid $S$.}
\end{figure}


\begin{theorem}\label{thm:construct-1-solenoid}
Let $M$ be a compact smooth manifold, and let $a\in H_1(M,{\RR})$
be a non-zero $1$-homology class. If $\dim M\geq 3$ then (a
positive multiple of) $a$ can be fully represented by an embedding
(of class $C^{\infty,2-\epsilon}$)
of the (oriented, uniquely ergodic) $1$-solenoid $S$ into $M$. If
$\dim M=2$ then (a positive multiple of) $a$ can be fully
represented by a transversal immersion of $S$ into $M$.
\end{theorem}

\begin{proof}
Let $C_1,\ldots, C_{b_1}$ be (integral) $1$-cycles which form a basis of the
(real) $1$-homology of $M$. Switch orientations and reorder the
cycles if necessary so that there are real numbers
$\lambda_1,\ldots, \lambda_r>0$ such that
 $$
 a=\lambda_1 C_1 + \cdots+ \lambda_r C_r.
 $$
By dividing by $\sum \lambda_i$ if necessary, we can assume that
$\sum \lambda_i=1$.

Consider the solenoid $S$ constructed above and partition the
cantor set $K$ into $r$ disjoint compact subsets $K_1,\ldots, K_r$
in cyclic order, each of which with
 $$
 \mu_K(K_i)=\lambda_i\, .
 $$
Consider the transversal $T=\{0\}\x\TT$ in $S_h$. We consider angles
$\theta_1,\theta_2,\ldots, \theta_n\in \TT$ in the same cyclic order as the
$K_i$, such that $K_i$ is contained in the open subset $U_i\subset
T$ with boundary points $\theta_i$ and $\theta_{i+1}$ (denoting
$\theta_{n+1}=\theta_1$). We may assume that $\theta_1=0$. Remove the segments
$[c,1]\x \{\theta_i\}$ from $S_h$ to get the open $2$-manifold
 $$
 U=S_h - \cup_i ([c,1]\x \{\theta_i\})\, .
 $$

\begin{figure}[h] \label{figure2}
\centering
\resizebox{6cm}{!}{\includegraphics{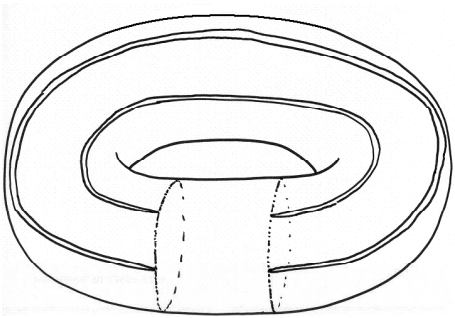}}    
\caption{The open manifold $U$.}
\end{figure}

By construction, our solenoid $S$ is included as a subset of $U$,
$S\subset U$.

Suppose that $\dim M\geq 3$. Then we can $C^\infty$-smoothly embed
$F:U\to M$ as follows: suppose that all cycles $C_i$ share a common
base-point $p_0\in M$ (and are otherwise disjoint to each other).
Then embed the central part $(0,c)\x \TT\subset U$ in a small ball
$B$ around $p_0$ and embed each of the $[c,1]\x U_i$ in $M-B$ in
such a way that if we contract $B$ to $p_0$ then the images of
$[c,1]\x \{t\}$, $t\in U_i$, represent cycles homologous to $C_i$.

The embedding $f$ of $S$ into $M$ is defined as the composition
$S\inc U \stackrel{F}{\to} M$. By theorem
\ref{thm:11.12}, as $S$ is uniquely ergodic,
to prove that $f:S\to M$
fully represents $a$, it is enough to see that $[f,S_{\mu}]=a$.

Let $\alpha$ be any closed $1$-form on $M$. Since $H^1(M)=H^1(M,B)$,
we may assume that $\alpha$ vanishes on $B$. We cover the solenoid
$S$ by the flow-boxes $((0,c)\x \TT)\cap S$ and $[c,1]\x K_i$,
$i=1,\ldots, r$. As $f^*\alpha$ vanishes in the first flow-box, we
have
  \begin{align*}
  \la [f,S_\mu],[\alpha]\ra &= \sum_{i=1}^r \int_{K_i} \left( \int_{[c,1]}
  f^*\alpha \right) d\mu_{K_i}(y) = \sum_{i=1}^r \int_{K_i} \la
  C_i ,[\alpha] \ra d\mu_{K_i}(y) \\ &= \sum_{i=1}^r \la
  C_i ,[\alpha] \ra \mu (K_i) = \sum_{i=1}^r \lambda_i \la
  C_i ,[\alpha] \ra =\la a,[\alpha]\ra\, ,
  \end{align*}
proving that $[f,S_\mu]=a$.

\medskip

Now suppose that $\dim M=2$. Let us do the appropriate modifications
to the previous construction. Choose cycles $C_i$ sharing a common
base-point $p_0\in M$, and such that their intersections (and
self-intersections) away from $p_0$ are transversal. Changing $C_i$
by $2C_i$ if necessary, we suppose that going around $C_i$ does not
change the orientation (that is, the normal bundle to $C_i$ is
oriented, hence trivial). From the manifold $U$ in Figure 2, remove $[0,c]\x
\{\t_1\}$ to get the open $2$-manifold
 $$
 V=\big( (0,c)\x (0,1) \big) \bigcup \cup_i \big( [c,1]\x U_i \big) \, .
 $$

\begin{figure}[h] \label{figure3}
\centering
\resizebox{5cm}{!}{\includegraphics{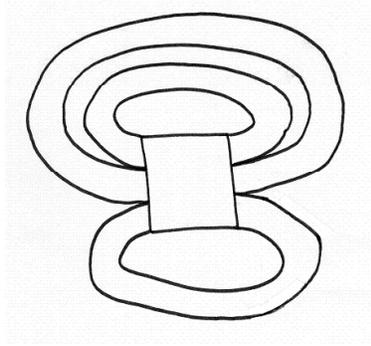}}    
\caption{The open manifold $V$}
\end{figure}

The manifold $V$ can be immersed into the surface $M$, $F:V\to M$,
in such a way that $(0,c)\x (0,1)$ is sent to a ball $B$ around
$p_0$, $[c,1]\x U_i$ are sent to $M-B$, the images of $[c,1]\x
\{t\}$, $t\in U_i$, represent cycles homologous to $C_i$ if we
contract $B$ to a point, and the intersections and
self-intersections of horizontal leaves are always transverse.

Note that the solenoid $S$ is not contained in $V$, since we have
removed $[0,c]\x\t_1$ from $U$. So we cannot define an immersion
$f:S\to M$ by restricting that of $F$. To define $f$ in $S\cap
((0,c)\x \TT)$, we need to explicit out our isotopy $h_t$.
Consider $h:\TT\to\TT$ and lift it to $\tilde{h}:\RR \to \RR$ with
$r:=\tilde{h}(0)\in (0,1)$. Consider a smooth function
$\rho:\RR\to [0,1]$, with $\rho(t)=1$ for $t\leq 0$, $\rho(t)=0$
for $t\geq c$, and $\rho'(t)<0$ for $t\in (0,c)$. Then we can
define
 $$
 h_t(x)= \tilde{h} ( \tilde{h}^{-1}(x) \rho(t) + x (1-\rho(t))) \mod \ZZ\, .
 $$

Define the immersion $f:S\to M$ as follows: $f$ equals $F$ for
$(t,x)\in [c,1]\x K \subset V$. For $(t,h^{-1}(h_t(x)))\in S\cap
([0,c]\x \TT)$, we set
  $$
  f(t,h^{-1}(h_t(x)))=\left\{\begin{array}{ll}
   F(t, (\tilde{h}^{-1}(x)+1) \rho(t) + x (1-\rho(t))), \qquad &
   x \in K \cap (0,r)\, ,\\
  F(t, \tilde{h}^{-1}(x) \rho(t) + x (1-\rho(t))), \qquad &
  x \in K \cap (r,1)\, .
   \end{array}\right.
  $$
It is easily checked that $f$ sends $S\cap ([0,c]\x \TT)$ into the
ball $B$ and the intersections of the leaves in this portion of the
solenoid are transverse.

The proof that
the Ruelle-Sullivan homology class of $f:S\to M$ is $[f,S_\mu]=a$ goes as before.
\end{proof}

\begin{remark}\label{rem:no-precisa-compactness}
We  do not need $M$ to be compact for the above construction to work. If
$M$ is non-compact, take integer $1$-cycles $C_1,C_2,\ldots$
(possibly infinitely many) which form a basis of $H_1(M,\RR)$. Then
for any $a\in H_1(M,\RR)$ there exist an integer $r\geq 1$ and
$\lambda_1,\ldots,\lambda_r\in \RR$ with $a=\sum \lambda_i C_i$. The
construction of theorem \ref{thm:construct-1-solenoid} works.

The solenoid $S$ is oriented, regardless of $M$ being oriented or
not.
\end{remark}

\section{Realization of $H_k(M,\RR)$}\label{sec:k-solenoids}

Let $M$ be a smooth compact oriented Riemannian $C^\infty$ manifold
and let $a\in H_k(M,\RR)$ be a non-zero real $k$-homology class.
We are going to construct
a uniquely ergodic $k$-solenoid
$f:S\to M$ with a $1$-dimensional transversal structure, immersed in
$M$ and fully representing $a$.

%
%

\medskip

To start with, fix a collection of compact $k$-dimensional smooth
oriented manifolds $S_1,\ldots, S_r$ and positive numbers
$\lambda_1,\ldots, \lambda_r>0$ such that $\sum \lambda_i=1$. For any fixed $\epsilon > 0$, let
$h:\TT\to \TT$ be a diffeomorphism of the circle which is a Denjoy
counter-example with an irrational rotation number and of class
$C^{2-\epsilon}$. Hence $h$ is uniquely
ergodic. Let $\mu_K$ be the unique invariant probability measure,
which is supported in the minimal Cantor set $K\subset \TT$. Partition the
Cantor set $K$ into $r$ disjoint compact subsets $K_1,\ldots, K_r$
in cyclic order, each of which with $\mu_K(K_i)=\lambda_i$.

We fix two points on each manifold $S_i$, and remove two small
balls, $D_i^+$ and $D_i^-$, around them. Denote
 $$
 S_i'=S_i- (D_i^+ \cup D_i^-)\, ,
 $$
so that $S_i'$ is a manifold with oriented boundary $\bd S_i'=\bd
D_i^+\sqcup \bd D_i^-$. Fix two diffeomorphisms: $\bd D_i^+ \cong
S^{k-1}$, being orientation preserving, and $\bd D_i^- \cong S^{k-1}$,
being orientation reversing. There are inclusions
 $$
  A_\pm :=\bigsqcup (\bd D_i^\pm \x K_i)
  \stackrel{i_\pm}{\inc} S^{k-1}\x S^1\, ,
 $$
whose image is $S^{k-1}\x K\subset S^{k-1}\x S^1$. Define
  $$
   S= \bigsqcup (S_i' \times K_i)_{/ (x,y) \sim i_+^{-1}\circ (\id\x h)
   \circ i_-(x,y),\,
   (x,y)\in A_- } \, .
  $$

This is an oriented $k$-solenoid of class $C^{\infty,2-\epsilon}$,
with $1$-dimensional transversal dimension. As $S^{k-1}\x K\subset
S$ in an obvious way, fixing a point $p\in S^{k-1}$, we have a
global transversal $T=\{p\}\x K \subset S^{k-1}\x K\subset S$.
Identifying $T\cong K$, the holonomy pseudo-group is generated by
$h:K\to K$. Hence $S$ is uniquely ergodic. Let $\mu$ denote the
transversal measure corresponding to $\mu_K$.

We want to give an alternative description of $S$. Fix an isotopy
$h_t$, $t\in [0,1]$, from $\id$ to $h$. Define the set
 $$
  W' : = \{ (t,x,h^{-1}(h_t(y)))\, ;\, t\in [0,1], x\in S^{k-1}, y\in K
   \}\subset [0,1]\x S^{k-1} \times S^1\, .
 $$

Then we have that
  $$
   S= \left( \bigsqcup (S_i' \times K_i) \sqcup W'
   \right)_{/ (x,y) \sim (0,i_-(x,y)),\, (x,y)\in \bd D_i^- \x K_i
  \atop (x,y)\sim (1,i_+(x,y)), \, (x,y)\in \bd D_i^+ \x K_i} \, .
  $$
Strictly speaking, we should say that they are diffeomorphic, but we
shall fix an identification. We define a map $\pi:S\to \TT$ by
  \begin{equation*}
  \left\{ \begin{array}{ll}
  \pi(t,x,h^{-1}(h_t(y))) =t-\frac12, \qquad & (t,x,h^{-1}(h_t(y)))\in W'\, ,\\[10pt]
  \pi(p)= \frac12, \qquad & p\in S-W'\, .
  \end{array} \right.
  \end{equation*}
Then $W=\Int(W')=\pi^{-1}(-\frac12,\frac12)$ is a trapping region according to
definition \ref{def:trapping}.

\medskip

Consider angles $\t_1,\t_2,\ldots, \t_n\in \TT$ in the same cyclic
order as the $K_i$, such that $K_i$ is contained in the open subset
$U_i\subset T$ with boundary points $\t_i$ and $\t_{i+1}$ (denoting
$\t_{n+1}=\t_1$). We may assume that $\t_1=0$. Then the solenoid $S$
sits inside the $(k+1)$-dimensional open manifold
  $$
  X= \bigsqcup (S_i'\times U_i) \sqcup ([0,1]\x S^{k-1} \times S^1
  )_{/ (x,y) \sim (0,i_-(x,y)),\, (x,y)\in \bd D_i^- \x U_i
  \atop (x,y)\sim (1,i_+(x,y)), \, (x,y)\in \bd D_i^+ \x U_i} \, ,
  $$
as the collection of points $(x,y)$, $x\in S_i'$, $y\in K_i$,
together with the points $(t,x,h^{-1}(h_t(y)))$, $x\in S^{k-1}$,
$y\in K$, $t\in [0,1]$.

\begin{figure}[h]\label{figure4}
\centering
\resizebox{10cm}{!}{\includegraphics{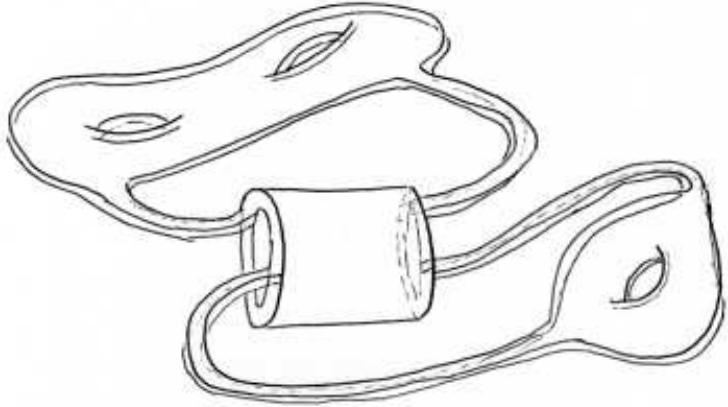}}    
\caption{The manifold $X$.}
\end{figure}

\begin{figure}[h]\label{figure5}
\centering
\resizebox{10cm}{!}{\includegraphics{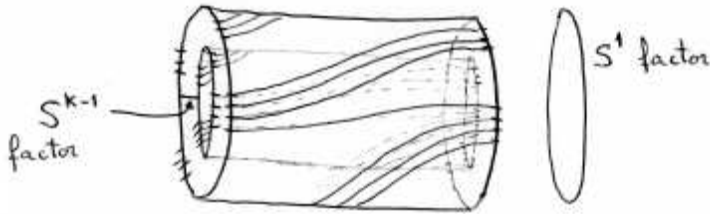}}    
\caption{The trapping region $W'$.}
\end{figure}


\begin{remark}
The $1$-solenoid constructed in section \ref{sec:1-solenoids}
corresponds to the case $S_i=S^1$, $i=1,\ldots, r$.
\end{remark}

\begin{theorem}\label{thm:construct-k-solenoid}
Let $M$ be a compact oriented smooth Riemannian manifold of
dimension $n$, and let $a\in H_k(M,{\RR})$ be a non-zero real
$k$-homology class. Then (a positive multiple of) $a$ can be fully represented by a
transversal immersion $f:S\to M$ of a uniquely ergodic oriented $k$-solenoid. If moreover,
$n\geq 2k+1$ then we can suppose that $f$ is an embedding.
\end{theorem}

\begin{proof}
By proposition \ref{prop:appendix-c:thom}, we may take a collection
$C_1,\ldots, C_{b_k}\in H_k(M,{\ZZ})$ which are a
basis of $H_k(M,{\QQ})$ and such that $C_i$ are represented by immersed
submanifolds $S_i\subset M$ with trivial normal bundle and self-transverse
intersections, and such that $S_i$ intersects $S_j$ transversally.
Moreover, if $n\geq 2k+1$, we may assume that there are no
either intersections or self-intersections.

After switching the orientations of $C_i$ if necessary,
reordering the cycles and multiplying $a$ by a suitable positive
real number, we may suppose that
 $$
 a=\lambda_1 C_1 +\ldots +\lambda_r C_r,
 $$
for some $r\geq 1$, $\lambda_i>0$, $1\leq i\leq r$, and $\sum
\lambda_i=1$. We construct the solenoid $S$ with the procedure above
starting with the manifolds $S_i$ and coefficients $\lambda_i$. This
is a uniquely ergodic $k$-solenoid with a $1$-dimensional
transversal structure, and a trapping region $W\subset S$ (see \cite{MPM3}).

Now we want to define an immersion $f:S\to M$, and to prove that
it fully represents $a$. We have the following cases:

\begin{enumerate}
\item $n\geq 2k+1$. The general position property on the $S_i$
implies that all $S_i$ are disjoint submanifolds of $M$. As the
normal bundle to $S_i$ is trivial and $U_i$ is an interval, we can
embedded $S_i\x U_i$ in a small neighbourhood of $S_i$.

Fix a base point $p_0\in M$ off all $S_i$. Take a small box
$B\subset M$ around $p_0$ of the form $B=[0,1]\x D^{n-1}$, where
$D^{n-1}$ is the open $(n-1)$-dimensional ball. Consider a circle
$S^1\subset D^{k+1}\subset D^{n-1}$ and let $D^k\x S^1 \subset
D^{k+1}\subset D^{n-1}$ be a tubular neighbourhood of it, with
boundary $S^{k-1}\x S^1$.

For each $i=1,\ldots, r$, fix $y_i\in U_i$, and consider two paths
in $M-\Int(B)$, say $\g_i^\pm$, where $\g_i^-$ goes from the point
$(0,y_i)\in \{0\}\x U_i \subset \{0\}\x S^1\subset \{0\}\x
D^{n-1}\subset B$ to the point $(p_i^-,y_i) \in S_i\x U_i$, and
$\g_i^+$ goes from $(1,y_i)\in \{1\}\x U_i \subset \{1\}\x
S^1\subset \{1\}\x D^{n-1}\subset B$ to $(p_i^+,y_i) \in S_i\x U_i$.
We arrange that $\g_i^\pm$ are transverse to $S_i\x U_i$ at
$(p_i^\pm,y_i)$ and are disjoint from all $S_j$ otherwise.

We thicken $\g_i^\pm$ to immersions $\g_i^\pm \x D^{k}\x U_i$ into
$M-\Int(B)$ such that one extreme goes to $D_i^\pm\x U_i$ and the
other goes to either $D^{k} \x U_i \x\{0\} \subset D^{k}\x S^1
\x\{0\} \subset D^{n-1}\x\{0\} \subset B$ for $\g_i^-$, or $D^{k} \x
U_i \x\{1\} \subset D^{k}\x S^1 \x\{1\} \subset D^{n-1}\x\{1\}
\subset B$ for $\g_i^-$. It is possible to do this in such a way
that the $U_i$ directions match, since $n\geq k+2$.

Recall that $S_i'=S_i - (D_i^+\cup D_i^-)$, and set
 \begin{equation*} 
  S_i''= S_i' \cup (\g_i^+ \x S^{k-1}) \cup (\g_i^- \x S^{k-1} )\,,
 \end{equation*}
which is diffeomorphic to $S_i'$ (to be rigorous, we should smooth
out corners). Then we can define the set
 \begin{eqnarray*}
 U &:=& \bigcup ((S_i'\x U_i) \cup ( \g_i^+ \x S^{k-1} \x U_i) \cup \\
 & & \cup ( \g_i^- \x S^{k-1} \x U_i) \cup
 ([0,1]\x S^{k-1} \x S^1)\, ,
 \end{eqnarray*}
which is a $(k+1)$-dimensional open manifold embedded in $M$. The manifold
$U$ is foliated as follows: $S_i''\x U_i$ is foliated by
$S_i''\x\{y\}$, for $y\in U_i$, and $[0,1]\x S^{k-1} \x S^1$ is
foliated by
 $$
 L_y=\{(t, x, h^{-1}(h_t(y))) \ ; \ t\in [0,1], x\in S^{k-1} \} \, ,
 $$
for $y \in S^1$. Clearly the solenoid $S$ is a sub-solenoid of $U$,
$S\subset U$. Restricting the embedding $F:U\to M$ to $S$ we get an
embedding $f:S\to M$.

By construction $f(W)\subset \Int(B)$, i.e. the image of the
trapping region is contained in a contractible ball.

 \item $1<n-k\leq k$.
 The same construction as in (1) works now,
 with the modification that we have to allow
 intersections of different leaves, but we may take them to be always transversal.
 So we get a transversal immersion $f:S\to M$.

%

 \item $n-k=1$. The submanifolds $S_i$ have trivial normal bundle and
 they intersect each other transversally. We cannot avoid
 that the paths $\g_i^\pm$ intersect other $S_j$, but we arrange these
 intersections to be transverse. This produces a transversal
 immersion $f$ of the region $S-W$ of the solenoid into $M-\Int(B)$.

 We have to modify the previous construction of the immersion of $W$ into $B$,
 as codimension one does not leave enough room for it to work. Consider the box $B=[0,1]\x
 D^{n-1}$ and remove the axis $A=[0,1]\x \{0\}$. Use polar
 coordinates to identify $B-A= [0,1] \x S^{k-1}\x (0,1)$, where the
 third coordinate corresponds to the radius.
 By construction, $W'\subset S$ embeds into $C=[0,1]\x S^{k-1} \x
 S^1$, as the set of points $(t,x,h^{-1}(h_t(y)))$, $t\in [0,1]$, $x\in S^{k-1}$
 and $y\in K$. We remove $D=[0,1]\x S^{k-1}\x \tau_1$ from $C$, so that $C-D=[0,1] \x S^{k-1}\x
(0,1)$. Then $W'$ immerses into $C-D$, by using the process at the
end of the proof of theorem \ref{thm:construct-1-solenoid} (now
there is an extra factor $S^{k-1}$ which plays no role). This is a
transversal immersion.

There is one extra detail that we should be careful about. When
connecting $p_i^\pm$ with the two faces of $B$, the orientations of
the $U_i$ should match. This happens because the normal bundle to
$S_i$ is trivial, and in this case $S_i\x U_i$ is (diffeomorphic to)
the normal bundle to $S_i$.
\end{enumerate}

We prove now that $f:S\to M$ fully represents
$a$, we use theorem \ref{thm:11.12}. The solenoid $S$ has a trapping
region $W$, and $f(W)\subset \Int(B)$, a contractible ball in $M$.
So we only need to see that $[f,S_\mu]=a$.

Recall that the
associated transversal measure is $\mu_K$ on the transversal $K$.
Let $\alpha$ be any closed $1$-form on $M$. Since $H^1(M)=H^1(M,B)$,
we may assume that $\alpha$ vanishes on $B$. We cover the solenoid
$S$ by the flow-boxes $S_i''\x K_i$, $i=1,\ldots, r$, and $W'$
(where the form $\alpha$ vanishes). Thus
  \begin{align*}
  \la [f,S_\mu],[\alpha]\ra &= \sum_{i=1}^r \int_{K_i} \left( \int_{S_i''}
  f^*\alpha \right) d\mu_{K_i}(y) = \sum_{i=1}^r \int_{K_i} \la
  C_i ,[\alpha] \ra \, d\mu_{K_i}(y) \\ &= \sum_{i=1}^r \la
  C_i ,[\alpha] \ra \mu (K_i) = \sum_{i=1}^r \lambda_i \la
  C_i ,[\alpha] \ra =\la a,[\alpha]\ra\, ,
  \end{align*}
proving that $[f,S_\mu]=a$.
\end{proof}

\begin{remark}
A similar comment to that of remark \ref{rem:no-precisa-compactness}
applies to the present situation, that is, the compactness of $M$ is
not necessary.
\end{remark}

\begin{remark}
The orientability of $M$ is not necessary as well. If $M$ is
non-orientable, we may consider its oriented double cover
$\pi:\tilde{M}\to M$. Then for $a\in H_k(M,\RR)$, there exists
$\tilde a \in H_k(\tilde M, \RR)$ with $\pi_*(\tilde a)=a$.

We can consider immersed submanifolds $f_i:S_i\inc \tilde M$ with
transversal self-inter\-sections, and intersecting transversally each other.
Then it is easy to perturb $f_i$ so that $\tilde{f}_i=\pi\circ f_i:S_i \to M$
are immersed oriented submanifolds with transversal self-intersections, and intersecting
transversally each other. This will allow to construct
a uniquely-ergodic oriented $k$-solenoid
$f:S\to M$ transversally immersed in $\tilde M$ fully representing (a multiple of)
$\tilde a$ such that $\pi\circ f: S\to M$
is transversally immersed in $M$ and fully represents (a multiple of) $a$.

If $n\geq 2k+1$, then we can assume that $f$ is an embedding (since transversal intersections
in this dimension do not happen).
\end{remark}

\bigskip

\begin{remark}\label{rem:Sullivan}
 In the article \cite{MPM-p}, we
prove that the currents that we have constructed are general enough in order to
fill a dense subset of the space of currents. Therefore, the generalized
Ruelle-Sullivan currents  associated to
immersed measured oriented uniquely-ergodic solenoids are dense
in the space of closed currents. This question was prompted to the authors
by Dennis Sullivan.
\end{remark}

\setcounter{section}{0}

\renewcommand{\thesection}{{Appendix}}
\section{Homology classes represented by submanifolds}
\renewcommand{\thesection}{\Alph{section}}\label{sec:appendix-c-thom}

By a theorem of Thom (see \cite{Thom1} and \cite{Thom2}), if $a\in
H_k(M,\ZZ)$ then there exists $N>>1$ such that $N\cdot a$ is
represented by a smooth submanifold of $M$. This submanifold
$C\subset M$ is oriented because it represents a non-zero homology
class (the top homology of a compact connected non-orientable
manifold is zero). Moreover, if $n\geq 2k+1$ or $n-k$ is odd then it
can be arranged that the normal bundle of $C$ is trivial
\cite{Thom1} \cite{Thom2}. If $n-k$ is even then it can be arranged
that the normal bundle is trivial if and only if $a\cup a=0$. Also
according to Sullivan \cite{Sul}, using Thom's method and the thesis
of Wells \cite{We} one can always represent $N\cdot a$ by an
immersed submanifold $f:C\to M$ with trivial normal bundle. (Note
that the normal bundle is defined for any immersed manifold.)
Moreover, with a small perturbation, we may assume that $f$ has only
transversal self-intersections.

For completeness, we give here a proof of these results by
elementary methods. We start first with the case of odd codimension.

\begin{lemma}\label{lem:C1}
Let $M$ be a compact and oriented manifold of dimension $n$. Let
$0\leq k\leq n$ with $n-k$ odd and $a\in H_k(M,\ZZ)$.

There exists $N>>1$ (dependent only on $n$ and $k$) and a smooth map
$f:M\to S^{n-k}$ such that for a generic point $p\in S^{n-k}$,
 $$
 C=f^{-1}(p) \subset M
 $$
is a smooth submanifold with trivial normal bundle such that $[C] =N
\cdot a$.
\end{lemma}

\begin{proof}
Let $\hat a\in H^{n-k}(M,\ZZ)$ be the Poincar\'e dual of $a$. We aim
to construct a map $f: M\to S^{n-k}$ such that $f^*([S^{n-k}])$ is a
multiple of $\hat a$. For this, consider a CW decomposition of $M$,
and let $\bar a \in C^{n-k}(M, \ZZ)$ with $\partial \bar a =0$ and
$[\bar a]=\hat a$.

We start by considering a map $f$ from the $(n-k-1)$-skeleton of
$\bar a$ to a base point $p\in S^{n-k}$. To define $f$ in the
$(n-k)$-skeleton, write
 $$
 \hat a=\sum_i n_i \, C_i^* \, ,
 $$
with $C_i$ being the $(n-k)$-cells of $M$. Then define $f_{|C_i}$ in
such a way that the induced map $f_{|C_i}:C_i/\partial C_i \to
S^{n-k}$ has degree $n_i$.

To extend $f$ to the higher skeleta, we work as follows: let $T$ be
an $(n-k+1)$-cell of $M$. Since
  $$
  \hat{a} (\partial T) = \partial\hat a (T)=0\, ,
  $$
we have that $f_{|\partial T}:\partial T \to S^{n-k}$ has degree
$0$. Therefore,  we can extend $f$ to a map $T \to S^{n-k}$. Now by
induction on $l=1,2,\ldots$ we assume that the map $f$ has been
extended to the $(n-k+l-1)$-skeleton of $M$ and we wish to extend it
to the $(n-k+l)$-skeleton. Let $T$ be a $(n-k+l)$-cell. The map
$f_{|\partial T }:\partial T\to S^{n-k}$ gives, recalling that
$\partial T \cong S^{n-k+l-1}$, an element
 $$
 [f_{|\partial T}] \in \pi_{n-k+l-1}(S^{n-k})\,.
 $$
By a result of Serre \cite{Se}, this group is torsion (since $n-k$
is odd). So there is a non-zero integer $k_l$ such that $k_l \cdot
[\partial T]=0$. Multiplying $a$ by $m_l$, the map $f'$ (in the
$(n-k+l-1)$-skeleton) corresponding to $a'=m_l\cdot a$ is the
composition of $f$ with a map $S^{n-k}\to S^{n-k}$ of degree $k_l$.
Therefore $[f'_{|\partial T}]=m_l\cdot [f_{|\partial T}]=0$, and
there is no obstruction to extend $f'$ to the cell $T$, and hence to
the $(n-k+l)$-skeleton.

In this way, we get an extension to the $n$-skeleton, i.e. to $M$.
This gives a continuous map $f: M\to S^{n-k}$ and it is trivial to
verify that $f^*([S^{n-k}])=N\cdot \hat a$, for some large integer
$N$ (actually, $N=m_2m_3\cdots m_{k}$).

Now, we homotop $f$ to a smooth function, which we call $f$ again.
Taking a regular value $p\in S^{n-k}$, we have a smooth submanifold
$C=f^{-1}(p)$ of dimension $k$, and with trivial normal bundle.
Clearly, $[C]=PD[N\cdot \hat a]= N\cdot a$.
\end{proof}

\begin{lemma}\label{lem:C2}
Let $M$ be a compact and oriented manifold of dimension $n$. Let
$1\leq k\leq n$ with $n-k$ even and $a\in H_k(M,\ZZ)$.

There exists $N>>1$ (only dependent of $n$ and $k$), an immersion
$i:C\to M$ of an oriented compact manifold $C$ with $i_*[C]=N \cdot
a$ and whose normal bundle $\nu_{C/M}\to C$ is trivial.
\end{lemma}

\begin{proof}
We consider $M\x \RR$, which is an $(n+1)$-manifold. It is open, but
the proof of lemma \ref{lem:C1} works for it and for the homology
class $a\in H_k(M\x \RR,\ZZ)\cong H_k(M,\ZZ)$. Note that $(n+1)-k$
is odd, so lemma \ref{lem:C1} guarantees the existence of a smooth
$k$-dimensional submanifold $C\subset M\times \RR$ with trivial
normal bundle, and such that $[C]=N\cdot a$, for some $N\geq 1$.

Denote by $j:C\hookrightarrow M\times \RR$ the inclusion, and let
$\pi: M\times \RR \to M$ be the projection into the first factor.
Denote by $t$ is the coordinate of the $\RR$ direction, and by
$\frac{\bd}{\bd t}$ the vertical vector field. Fixing a non-zero
normal vector field $X$ to $C\subset M\x \RR$, the compression
theorem in \cite{RoSa} allows to isotop the pair $(j,X)$ to
$(j',\frac{\bd}{\bd t})$, where $j': C\hookrightarrow M\x \RR$ is an
embedding and $\frac{\bd}{\bd t}$ becomes a normal vector field to
$j'(C)$. Therefore the composition $i=\pi\circ j': C\to M$ is an immersion.
Clearly, $i_*[C]=\pi_*j'_*[C] =\pi_*[C] = \pi_*(N \cdot
a) = N \cdot a\in H_k(M, \RR )$ and the normal bundle to $C$ in $M$
is trivial.
\end{proof}

The precise result that we use in section \ref{sec:k-solenoids} is the
following:

\begin{proposition}\label{prop:appendix-c:thom}
Let $M$ be a compact manifold of dimension $n$, and let $b_k=\dim
H_k(M,\RR)$. Then we may take a collection $C_1,\ldots, C_{b_k}\in
H_k(M,{\ZZ})$ which forms a basis of $H_k(M,{\QQ})$ and such that
$C_i$ are represented by immersed submanifolds $S_i\subset M$ with
trivial normal bundle and self-transverse intersections, and such
that $S_i$ intersects $S_j$ transversally. Moreover, if $n\geq
2k+1$, we may assume that there are no either intersections or
self-intersections.
\end{proposition}

\begin{proof}
 Using lemma \ref{lem:C1} or lemma \ref{lem:C2} (according to the
 parity or $n-k$),  we may find a collection of
 immersed oriented compact
 submanifolds $S_i$ with trivial normal bundle representing
 a basis for the rational homology $H_k(M,{\QQ})$.

 Now a small perturbation of each $S_i$ makes all intersections of $S_i$ with
 $S_j$, $i\neq j$, and all self-intersections of $S_i$, transverse. If
 $n\geq 2k+1$, the transversality of the intersections implies that there are
 no intersections at all. So the result follows.
\end{proof}

\end{document}